\documentclass[12pt]{article}
\usepackage{graphicx}
\usepackage{amsmath}
\usepackage{amsfonts}
\usepackage{amsthm}
\usepackage[T1]{fontenc}
\usepackage{url}
\usepackage{color}
\usepackage[margin=1in]{geometry}
\usepackage{amssymb,bm}
\usepackage{amsmath}
\usepackage{amsthm}
\usepackage{graphicx}
\usepackage[active]{srcltx} 
\usepackage{hyperref}
\hypersetup{pdfborder=0 0 0}

\newtheorem{theorem}{{\sc Theorem}}[section]

\newtheorem{lemma}[theorem]{{\sc Lemma}}

\newtheorem{remark}[theorem]{Remark}

\newcommand{\nth}[1]{\displaystyle\frac{1}{#1}}

\newcommand{\dif}[2]{\displaystyle\frac{\partial #1}{\partial #2}}
\newcommand{\Grad}{\nabla}

\def\XXint#1#2#3{{\setbox0=\hbox{$#1{#2#3}{\int}$ }
\vcenter{\hbox{$#2#3$ }}\kern-.6\wd0}}

\newcommand{\Gk}{\kappa}

\newcommand{\Gth}{\theta}

\bmdefine\BGa{\alpha}
\bmdefine\BGb{\beta}
\bmdefine\BGd{\delta}
\bmdefine\BGe{\epsilon}
\bmdefine\BGve{\varepsilon}
\bmdefine\BGf{\phi}
\bmdefine\BGvf{\varphi}
\bmdefine\BGg{\gamma}
\bmdefine\BGc{\chi}
\bmdefine\BGi{\iota}
\bmdefine\BGk{\kappa}
\bmdefine\BGl{\lambda}
\bmdefine\BGn{\eta}
\bmdefine\BGm{\mu}
\bmdefine\BGv{\nu}
\bmdefine\BGp{\pi}
\bmdefine\BGth{\theta}
\bmdefine\BGvth{\vartheta}
\bmdefine\BGr{\rho}
\bmdefine\BGvr{\varrho}
\bmdefine\BGs{\sigma}
\bmdefine\BGvs{\varsigma}
\bmdefine\BGt{\tau}
\bmdefine\BGj{\tau}
\bmdefine\BGu{\upsilon}
\bmdefine\BGo{\omega}
\bmdefine\BGx{\xi}
\bmdefine\BGy{\psi}
\bmdefine\BGz{\zeta}
\bmdefine\BGD{\Delta}
\bmdefine\BGF{\Phi}
\bmdefine\BGG{\Gamma}
\bmdefine\BGL{\Lambda}
\bmdefine\BGP{\Pi}
\bmdefine\BGT{\Theta}
\bmdefine\BGS{\Sigma}
\bmdefine\BGU{\Upsilon}
\bmdefine\BGO{\Omega}
\bmdefine\BGX{\Xi}
\bmdefine\BGY{\Psi}

\bmdefine\BCA{{\mathcal A}}
\bmdefine\BCB{{\mathcal B}}
\bmdefine\BCC{{\mathcal C}}
\bmdefine\BCD{{\mathcal D}}
\bmdefine\BCE{{\mathcal E}}
\bmdefine\BCF{{\mathcal F}}
\bmdefine\BCG{{\mathcal G}}
\bmdefine\BCH{{\mathcal H}}
\bmdefine\BCI{{\mathcal I}}
\bmdefine\BCJ{{\mathcal J}}
\bmdefine\BCK{{\mathcal K}}
\bmdefine\BCL{{\mathcal L}}
\bmdefine\BCM{{\mathcal M}}
\bmdefine\BCN{{\mathcal N}}
\bmdefine\BCO{{\mathcal O}}
\bmdefine\BCP{{\mathcal P}}
\bmdefine\BCQ{{\mathcal Q}}
\bmdefine\BCR{{\mathcal R}}
\bmdefine\BCS{{\mathcal S}}
\bmdefine\BCT{{\mathcal T}}
\bmdefine\BCU{{\mathcal U}}
\bmdefine\BCV{{\mathcal V}}
\bmdefine\BCW{{\mathcal W}}
\bmdefine\BCX{{\mathcal X}}
\bmdefine\BCY{{\mathcal Y}}
\bmdefine\BCZ{{\mathcal Z}}

\bmdefine\Bzr{ 0}
\bmdefine\Ba{ a}
\bmdefine\Bb{ b}
\bmdefine\Bc{ c}
\bmdefine\Bd{ d}
\bmdefine\Be{ e}
\bmdefine\Bf{ f}
\bmdefine\Bg{ g}
\bmdefine\Bh{ h}
\bmdefine\Bi{ i}
\bmdefine\Bj{ j}
\bmdefine\Bk{ k}
\bmdefine\Bl{ l}
\bmdefine\Bm{ m}
\bmdefine\Bn{ n}
\bmdefine\Bo{ o}
\bmdefine\Bp{ p}
\bmdefine\Bq{ q}
\bmdefine\Br{ r}
\bmdefine\Bs{ s}
\bmdefine\Bt{ t}
\bmdefine\Bu{ u}
\bmdefine\Bv{ v}
\bmdefine\Bw{ w}
\bmdefine\Bx{ x}
\bmdefine\By{ y}
\bmdefine\Bz{ z}
\bmdefine\BA{ A}
\bmdefine\BB{ B}
\bmdefine\BC{ C}
\bmdefine\BD{ D}
\bmdefine\BE{ E}
\bmdefine\BF{ F}
\bmdefine\BG{ G}
\bmdefine\BH{ H}
\bmdefine\BI{ I}
\bmdefine\BJ{ J}
\bmdefine\BK{ K}
\bmdefine\BL{ L}
\bmdefine\BM{ M}
\bmdefine\BN{ N}
\bmdefine\BO{ O}
\bmdefine\BP{ P}
\bmdefine\BQ{ Q}
\bmdefine\BR{ R}
\bmdefine\BS{ S}
\bmdefine\BT{ T}
\bmdefine\BU{ U}
\bmdefine\BV{ V}
\bmdefine\BW{ W}
\bmdefine\BX{ X}
\bmdefine\BY{ Y}
\bmdefine\BZ{ Z}

\begin{document}
\title{Gaussian curvature as an identifier of shell rigidity}
\author{D. Harutyunyan}
\maketitle

\begin{abstract}
In the paper we deal with shells with non-zero Gaussian curvature. We derive sharp Korn's first (linear geometric rigidity estimate) and second inequalities on that kind of shells for zero or periodic Dirichlet, Neumann, and Robin type boundary conditions. We prove that if the Gaussian curvature is positive, then the optimal constant in the first Korn inequality scales like $h,$ and if the Gaussian curvature is negative, then the Korn constant scales like $h^{4/3},$ where $h$ is the thickness of the shell. These results have classical flavour in continuum mechanics, in particular shell theory. The Korn first inequalities are the linear version of the famous geometric rigidity estimate by Friesecke, James and M\"uller for plates [\ref{bib:Fri.Jam.Mue.2}] (where they show that the Korn constant in the nonlinear Korn's first inequality scales like $h^2$), extended to shells with nonzero curvature. We also recover the uniform Korn-Poincar\'e inequality proven for "boundary-less" shells by Lewicka and M\"uller in [\ref{bib:Lew.Mul.2}] in the setting of our problem. The new estimates can also be applied to find the scaling law for the critical buckling load of the shell under in-plane loads as well as to derive energy scaling laws in the pre-buckled regime. The exponents $1$ and $4/3$ in the present work appear for the first time in any sharp geometric rigidity estimate.
\end{abstract}

\section{Introduction}
\label{sec:1}
It is known, that the rigidity of a shell under compression is closely related to the optimal Korn's constant in the nonlinear (in many cases linear) first Korn's inequality-a geometric rigidity estimate for $W^{1,2}$ fields under the appropriate Robin type boundary conditions (or without them [\ref{bib:Fri.Jam.Mue.1}]) arising from the nature of the compression, e.g., [\ref{bib:Gra.Tru.},\ref{bib:Gra.Har.2}]. In their celebrated work, Friesecke, James and M\"uller [\ref{bib:Fri.Jam.Mue.1},\ref{bib:Fri.Jam.Mue.2}] derived a geometric rigidity estimate for plates, which gave rise to derivation of a hierarchy of plate theories for different scaling regimes of the elastic energy depending on the thickness $h$ of the plate [\ref{bib:Fri.Jam.Mue.2}]. This type of theories have been derived by Gamma-convergence and rely on $L^p$-compactness arguments and of course the underlying nonlinear Korn's inequality, which plays the main role. While the rigidity, and in particular buckling of plates has been understood almost completely [\ref{bib:Fri.Jam.Mue.2},\ref{bib:Dau.Sur.},\ref{bib:Con.Mag.}], the rigidity and buckling of shells is less well understood. To be more precise, to our best knowledge only low energy scalings like $E\sim h^\alpha,$ where $\alpha\geq 3,$ (which roughly speaking corresponds to bending, [\ref{bib:Fri.Jam.Mor.Mue.},\ref{bib:Lew.Mor.Pak.}]), or very high energy profiles $E\sim h$ (which corresponds to stretching, where the limiting energy completely eliminates bending and captures only stretching of the shell) have been studied [\ref{bib:LeD.Rao.1},\ref{bib:LeD.Rao.2}]. We also refer to the papers [\ref{bib:Fri.Jam.Mor.Mue.},\ref{bib:Lew.Mor.Pak.},\ref{bib:Hor.Lew.Pak.},\ref{bib:Hor.Vel.},\ref{bib:Lew.Li.},\ref{bib:Lew.Mah.Pak.},\ref{bib:Lew.Pak.},
\ref{bib:LeD.Rao.1},\ref{bib:LeD.Rao.2}] for some results on shell deformation and theories. An interested reader can also check the work of Ciarlet for linearized rod, plate, shell and thin structures with junctions theories [\ref{bib:Ciarlet1},\ref{bib:Ciarlet2},\ref{bib:Ciarlet3},\ref{bib:Ciarlet4},\ref{bib:Ciarlet5},\ref{bib:Ciarlet6},\ref{bib:Cia.Rab.}].
The gap in nonlinear shell theories is due to the lack of sharp\footnote{Although the same inequality (\ref{1.1}) has been proven in [\ref{bib:Fri.Jam.Mor.Mue.}], it is not sharp for shells with a uniformly nonzero principal curvature as the present work and [\ref{bib:Gra.Har.4}] show.} rigidity estimates for shells and $L^p$-compactness. In particular, for cylindrical shells the sequence of buckling modes has the only weak limit zero as understood by Grabovsky and Harutyunyan in [\ref{bib:Gra.Har.3}], thus one has to look for a Young measure as the limit to capture the oscillations developed by the deformation sequences. In a series of papers [\ref{bib:Gra.Har.1},\ref{bib:Gra.Har.3},\ref{bib:Gra.Har.2}], Grabovsky and Harutyunyan studied the buckling of cylindrical shells under axial compression, where it became clear that even for simple geometries the problem may be unexpectedly rich and complex. It has been understood by Grabovsky and Truskinovsky in [\ref{bib:Gra.Tru.}], see also [\ref{bib:Gra.Har.2}], that in the case of shells, or simply any thin structures, when there is enough rigidity then actually the linear Korn's first inequality can replace the nonlinear one in the buckling problem. Let us recall some of the known results to make things clear. The rigidity estimate of Friesecke, James and M\"uller mentioned above asserts the following: \textit{Assume $\omega\subset\mathbb R^2$ is a bounded Lipschitz domain and let $\Omega=\omega\times[0,h]$ be a plate with base $\omega$ and thickness $h>0.$ Then there exists a constant $C=C(\omega),$ such that for every vector field $\Bu\in H^1(\Omega)$ there exists a constant rotation $\BR\in SO(3)$, such that }
\begin{equation}
\label{1.1}
\|\nabla\Bu-\BR\|^2\leq\frac{C}{h^2}\int_\Omega\mathrm{dist}^2(\nabla\Bu(x),SO(3))dx.
\end{equation}
The linear Korn's first inequality introduced by Korn [\ref{bib:Korn.1},\ref{bib:Korn.2}] and proven by many different authors in
different settings [\ref{bib:Con.Dol.Mul.},\ref{bib:Friedrichs},\ref{bib:Horgan},\ref{bib:Kon.Ole.1},\ref{bib:Kon.Ole.2},\ref{bib:Pay.Wei.},\ref{bib:Kohn}] reads a follows: \textit{Assume $\Omega\subset\mathbb R^n$ is a bounded Lipschitz domain and the closed subspace $V\subset H^1(\Omega,\mathbb R^n)$ does not contain a rigid body motion $Ax+b,$ $(A^T=-A)$ with $A\neq 0.$ Then, there exists a constant $C=C(\Omega,V)$ such that for every vector field $\Bu\in V$ there holds: }
\begin{equation}
\label{1.2}
 C(\Omega,V)\|\nabla\Bu\|^2\leq \|e(\Bu)\|^2,
\end{equation}
where $e(\Bu)=\frac{1}{2}(\nabla\Bu+\nabla\Bu^T)$ is the symmetrized gradient (the strain in linear elasticity).
The subspace $V$ is identified by the boundary conditions satisfied by the displacement $\Bu$ coming from the nature of the problem under consideration. It is known that then the dependence of the optimal constant $C(\Omega,V),$ called Korn's constant, on the geometric parameters of the domain $\Omega$ is crucial in the study of the buckling problem [\ref{bib:Gra.Tru.},\ref{bib:Gra.Har.3}], namely the formula (\ref{1.4}) below for the buckling load has been proven in [\ref{bib:Gra.Har.3}]. Recall, that if the body $\Omega$ is subject to dead load $\Bt(x),$ then the total energy of the deformation $\By(x)=\Bx+\Bu(x)$ is given by
$$E(\By)=\int_{\Omega}W(\nabla \By(x))dx-\int_{\partial\Omega}\By(s)\cdot \Bt(s)ds,$$
where $W$ is the elastic energy density of the body $\Omega.$ In their theory of buckling of thin structures, Grabovsky and Truskinovsky [\ref{bib:Gra.Tru.}] assume that the thin domain $\Omega=\Omega_h$ is subject to the dead load
$\Bt_\lambda(x)=\lambda \Bt(x),$ where $\lambda>0$ is the magnitude of the load. Then they introduce the Rayleigh quotient
\begin{equation}
\label{1.3}
R(h,\mathbf{\phi})=\frac{\int_{\Omega_{h}}(L_{0}e(\mathbf{\phi}),e(\mathbf{\phi}))dx}
{\int_{\Omega_{h}}(\sigma_{h},\nabla\mathbf{\phi}^{T}\nabla\mathbf{\phi})dx},
\end{equation}
which was studied further by Grabovsky and Harutyunyan in [\ref{bib:Gra.Har.3}].
Here, $L_0=W_{FF}(I)$ is the linear elasticity matrix, $\sigma_h$ is the Piola-Kirchhoff stress tensor and
$e(\phi)=\frac{1}{2}\left(\nabla\mathbf{\phi}+\nabla\mathbf{\phi}^T\right)$ is the linear elastic strain.
Define furthermore the constitutively linearized buckling load [\ref{bib:Gra.Har.3}] as
\begin{equation}
\label{1.4}
\lambda(h)=\inf_{\mathbf{\phi}\in V_h}R(h,\mathbf{\phi}),
\end{equation}
where here $V_h$ is the closed subspace of $H^1(\Omega_h,\mathbb R^3)$ identified by the boundary conditions fulfilled by the displacement $\phi.$
Then the following theorem has been proven in [\ref{bib:Gra.Har.3}]:
\begin{theorem}[The critical load]
Assume that the quantity $\lambda(h)$ defined in (\ref{1.4})  satisfies
$\lambda(h)>0$ for all sufficiently small $h$ and
\begin{equation}
 \label{1.5}
\lim_{h\to 0}\frac{{\lambda}(h)^{2}}{K(\Omega_h,V_{h})}=0,
\end{equation}
where $K(\Omega_h,V_h)$ is the Korn's constant of the domain $\Omega_h$ associated to the subspace $V_h.$
Then $\lambda(h)$ is the critical buckling load and the variational problem (\ref{1.3})
captures the buckling modes (deformations) too (see [\ref{bib:Gra.Har.3}] for a precise definition of buckling modes).
\end{theorem}
As long as the task is to solve the variational problem (\ref{1.4}), one clearly has to calculate the Korn's constant $C(\Omega_h,V_h)$
as the numerator and the denominator in (\ref{1.3}) are quadratic forms in the strain $e(\Bu)$ and the gradient $\nabla\Bu$ respectively.
All the Korn's constants calculated prior to the work of Grabovsky and Harutyunyan [\ref{bib:Gra.Har.1}], (where the authors calculate Korn's constant for cylindrical shells), scale either like $h^2$ or $O(1).$ The one for cylindrical shells calculated in [\ref{bib:Gra.Har.1}] scales like $h^{3/2},$ which gave rise to the following very important question: \textit{What geometric quantities make shells more rigid than plates and how can one compare the rigidity of two different shells? It turns out, that the answer to the above question is encoded in the principal and Gaussian curvatures of the shell: the bigger the curvatures are the more rigid the shell is as will be clear from the analysis in this work.} Mimicking cylindrical and conical shells, Grabovsky and Harutyunyan derived sharp Korn's first inequalities for shells with one principal curvature vanishing and the other one having a constant sign in [\ref{bib:Gra.Har.4}], where they prove that the Korn constant again scales like $h^{3/2}$ as for regular circular cylindrical shells. In this work we derive sharp Korn's first inequalities for shells of positive or negative Gaussian curvature. It turns out, that for negative curvature the scaling of Korn's constant is $h^{4/3}$ and for positive curvature it is $h$, i.e., the zero and positive Gaussian curvature shells are the least and  most rigid (under compression) ones respectively. All three exponents $3/2,$ $4/3$ and $1$ are completely new in Korn's inequalities\footnote{The exponent $3/2$ appeared first in the work of Grabovsky and Harutyunyan [\ref{bib:Gra.Har.1}]}.
Let us mention, that a hint for the exponents $4/3$ and $1$ comes from the book of Tovstik and Smirnov [\ref{bib:Tov.Smi.}], where the authors
construct different Ans\"atze for shells, without proving any Ansatz-free lower bounds. Korn's first inequalities mentioned above and proven in this work are the linear analogues of Friesecke, James and M\"uller rigidity estimate (\ref{1.1}). This work opens up new ideas in the shell theory and also makes it clear that when dealing with a shell, it may be most convenient to work in the local coordinates where the principal curvatures of the shell pop up in the gradient structure of the vector field $\Bu\colon\Omega\to\mathbb R^3$ explicitly. Although it is hard to predict, we believe the analogous nonlinear estimates (with or without boundary conditions) will have the same scaling of the constant $C$ in terms of the shell thickness. That new nonlinear rigidity estimates will naturally give rise to shell theories being derived from 3-dimensional elasticity by $\Gamma$-convergence like [\ref{bib:Fri.Jam.Mue.2},\ref{bib:Fri.Jam.Mor.Mue.},\ref{bib:Lew.Mor.Pak.},\ref{bib:Lew.Pak.}]. Let us now comment on the proof of our main results and also the relevant sharp rigidity estimates proven before. First of all it is worth mentioning that the only available constants (in terms of the thickness $h$) in the literature are $K(\Omega)=Ch^2$ and $K(\Omega)=C,$ where $C$ is a constant and $K(\Omega)$ is the Korn constant of the domain $\Omega$ (with or without boundary conditions), the constant in the rigidity estimate or in Korn's first inequality. The lower bound $K\geq Ch^2$ proven for plates by Friesecke, James and M\"uller in [\ref{bib:Fri.Jam.Mue.2}] is universal\footnote{It is satisfied even for rods, which seem to be less rigid.}, in the sense that the same authors and Mora prove it for shells implicitly in [\ref{bib:Fri.Jam.Mor.Mue.}], however, as mentioned before, it is not sharp for general shells. The main tool in their proof is the geometric rigidity estimate of Friesecke, James and M\"uller proven in [\ref{bib:Fri.Jam.Mue.1}], while the main technique is the localization technique, i.e., one divides the shell in many little cubes of size $h,$ proves a local estimate on each of them, and sums them up, thus the curvature of the shell is not taken into account and does not play a role. This technique gives the universal lower bound $K\geq Ch^2$ and evidently can not give better such as $Ch^\alpha,$ where $\alpha<2$ (also because of the simple reason of neglecting the curvature). Another interesting sharp Korn's first inequality is proven by Lewicka and M\"uller [\ref{bib:Lew.Mul.2}], where the authors prove a uniform Korn-Poincar\'e estimate $\|\Bu\|_{H^1(\Omega)}\leq C\|e(\Bu)\|_{L^2(\Omega)}$ for shells $\Omega,$ that are made around a smooth boundary-less surface $S\subset\mathbb R^3,$ and for tangential vector fields $\Bu\in H^1(\Omega)$ with $\Bu\cdot\Bn=0$ on $\partial\Omega.$ An interesting aspect of the present work is that we also recover the Lewicka-M\"uller uniform Korn-Poincar\'e inequality, of course in the setting of uniformly positive or negative curvature\footnote{Of course there is no boundary-less surface in $\mathbb R^3$ with uniformly negative Gaussian curvature. The statement for negative Gaussian curvature shells must be understood as for such shells with boundary (in the in-plane directions) and $H^1$ vector fields defined on them that satisfy zero or periodic boundary conditions in each of the principal directions (if the shell parametrization allows for talking about periodicity) on the thin faces of the shell. Roughly speaking, having no boundary is the same as having periodic boundary conditions.} (see Theorem~\ref{th:3.1} and Remark~\ref{rem:3.2}), which seems to replace their "contraction type Schwartz inequality" hypothesis (see [\ref{bib:Lew.Mul.2}, (2.4) in Theorem~2.1]), that seems to be curvature and thickness-change related. Also, when proving a Korn inequality for a shell, then bounding the normal "out-of-plane" component is typically the most complex one, while in the presence of tangential boundary conditions it becomes trivial by the Poincar\'e inequality in the normal direction. Regarding our strategy of the proof, it turns out that for both positive and negative Gaussian curvatures, one can prove a Korn-Poincar\'e inequality that controls the norm of the displacement components in the principal "in plane" directions by the norm of the symmetrized gradient $e(\Bu).$ Moreover, it turns out that for positive curvatures, one can bound the normal component of the displacement $\Bu$ by the norm of the symmetrized gradient $e(\Bu)$ as well. These are not possible for zero Gaussian curvature as the Ans\"atze in [\ref{bib:Gra.Har.1},\ref{bib:Gra.Har.4}] show. Then utilizing these bounds we simplify the problem by "throwing away" the parts of the gradient $\nabla\Bu$ that are of the order of the norm $e(\Bu).$ For the remaining part of the gradient we then prove a sharp first and a half Korn inequality (a Korn interpolation inequality introduced first in [\ref{bib:Gra.Har.1}]) with a constant scaling like $h.$ Finally, we combine the Korn-Poincar\'e inequality with the first-and-a-half Korn inequality to derive the sharp Korn's first inequality. As already mentioned, the Ansatz for negative Gaussian curvature is due to Tovstik and Smirnov and can be found in the book [\ref{bib:Tov.Smi.}], while the Ansatz for positive Gaussian curvature is Kirchhoff-like and is given by us. We mention in conclusion, that our analysis goes through for shells with nonconstant thickness as demonstrated in [\ref{bib:Harutyunyan.1},\ref{bib:Harutyunyan.2}].

\section{Preliminaries}
\setcounter{equation}{0}
\label{sec:2}

In this section we introduce notation and some definitions. In what follows we assume that the mid-surface $S$
of the shell under consideration is of class $C^3$ up to the boundary. Denote by $z$ and $\Gth$ the coordinates on the mid-surface
of the shell, such that $z=$constant and $\Gth=$constant are the principal lines. Here,
$\Gth$ will denote the circumferential and $z$--the longitudinal coordinates for
cylindrically shaped shells and in particular, in the case of a straight circular cylinder, $\Gth$ and $z$ are the standard cylindrical coordinates. Assume the mid-surface is given by the parametrization $\Br=\Br(\Gth,z).$ Then, introducing the normal coordinate $t$,
we obtain the set of orthogonal curvilinear coordinates $(t,\Gth,z)$ on the entire shell given by
\[
\BR(t,\Gth,z)=\Br(z,\Gth)+t\Bn(z,\Gth),
\]
where $\Bn$ is the \emph{outward} unit normal. In this paper we will study shells of constant thickness $h$ around $S,$ i.e., the domain
$$\Omega=\left\{S+t\Bn\ : \ t\in\left[-\frac{h}{2},\frac{h}{2}\right]\right\}.$$
As mentioned in Remark~\ref{rem:3.0}, the present results hold for any surface $S$ with nonzero Gaussian curvature and for displacements $\Bu$ satisfying some boundary conditions (depending on the geometry of the shell) on the thin faces of the shell $\Omega.$ However, in order to simplify the presentation of the analysis, we will assume in the sequel that the mid-surface $S$ is given by $S=\{(\Gth,z) \ : \ \Gth\in[0,\omega], \ z\in[1,1+l]\}$, and thus the shell $\Omega$ is given by
 \begin{equation}
  \label{2.1}
  \Omega=\left\{\BR(t,\Gth,z): t\in\left[-\frac{h}{2},\frac{h}{2}\right],\ \Gth\in[0,\omega], \ z\in[1,1+l]\right\},
\end{equation}
i.e., we assume that the shell is cut along the principal directions of the mid-surface. Denote next
\[
A_{z}=\left|\frac{\partial \Br}{\partial z}\right|,\qquad A_{\Gth}=\left|\frac{\partial \Br}{\partial\Gth}\right|,
\]
the two nonzero components of the metric tensor of the mid-surface and the two principal curvatures by $\Gk_{z}$ and $\Gk_{\Gth}$.
The signs of $\Gk_{z}$ and $\Gk_{\Gth}$ are choosen such that $\kappa_{z}$ and $\kappa_{\Gth}$ are positive for a sphere.
The four functions $A_{\Gth}$, $A_{z}$, $\Gk_{\Gth}$, and $\Gk_{z}$ satisfy the Codazzi-Gauss relations (see
e.g. [\ref{bib:Lee},\ref{bib:Tov.Smi.}]),
\begin{align}
  \label{2.2}
  &\dif{\Gk_{z}}{\Gth}=(\Gk_{\Gth}-\Gk_{z})\frac{A_{ z,\Gth}}{A_{ z}},\qquad
\dif{\Gk_{\Gth}}{ z}=(\Gk_{z}-\Gk_{\Gth})\frac{A_{\Gth, z}}{A_{\Gth}},\\ \nonumber
  &\dif{}{ z}\left(\frac{A_{\Gth, z}}{A_{ z}}\right)+
\dif{}{\Gth}\left(\frac{A_{ z,\Gth}}{A_{\Gth}}\right)=-A_{ z}A_{\Gth}\Gk_{z}\Gk_{\Gth},
\end{align}
and define the Levi-Civita connection on the mid-surface of the shell via the following derivation formulas
\[
\Grad_{\Be_{z}}\Be_{z}=-\nth{A_{ z}A_{\Gth}}\dif{A_{ z}}{\Gth}\Be_{\Gth}-\Gk_{z}\Bn,\qquad
\Grad_{\Be_{z}}\Be_{\Gth}=\nth{A_{ z}A_{\Gth}}\dif{A_{ z}}{\Gth}\Be_{z},\qquad
\Grad_{\Be_{z}}\Bn=\Gk_{z}\Be_{z},
\]
\[
\Grad_{\Be_{\Gth}}\Be_{\Gth}=-\nth{A_{ z}A_{\Gth}}\dif{A_{\Gth}}{ z}\Be_{z}-\Gk_{\Gth}\Bn,\qquad
\Grad_{\Be_{\Gth}}\Be_{z}=\nth{A_{ z}A_{\Gth}}\dif{A_{\Gth}}{ z}\Be_{\Gth},\qquad
\Grad_{\Be_{\Gth}}\Bn=\Gk_{\Gth}\Be_{\Gth}.
\]
Our convention is that inside the gradient matrix we will use $f_{,\alpha}$ for the partial derivative $\frac{\partial}{\partial\alpha}.$
Given  a weakly differentiable vector field $\Bu\colon\Omega\to\mathbb R^3$, using the above formulas for the partial derivatives,
we can compute the components of $\nabla\Bu$ in the orthonormal basis $\Be_{t}$, $\Be_{\Gth}$, $\Be_{z}$ to get
\begin{equation}
\label{2.3}
\nabla\Bu=
\begin{bmatrix}
  u_{t,t} & \dfrac{u_{t,\Gth}-A_{\Gth}\Gk_{\Gth}u_{\Gth}}{A_{\Gth}(1+t\Gk_{\Gth})} &
\dfrac{u_{t,z}-A_{z}\Gk_{z}u_{z}}{A_{z}(1+t\Gk_{z})}\\[3ex]
u_{\Gth,t}  &
\dfrac{A_{z}u_{\Gth,\Gth}+A_{z}A_{\Gth}\Gk_{\Gth}u_{t}+A_{\Gth,z}u_{z}}{A_{z}A_{\Gth}(1+t\Gk_{\Gth})} &
\dfrac{A_{\Gth}u_{\Gth,z}-A_{z,\Gth}u_{z}}{A_{z}A_{\Gth}(1+t\Gk_{z})}\\[3ex]
u_{ z,t}  & \dfrac{A_{z}u_{z,\Gth}-A_{\Gth,z}u_{\Gth}}{A_{z}A_{\Gth}(1+t\Gk_{\Gth})} &
\dfrac{A_{\Gth}u_{z,z}+A_{z}A_{\Gth}\Gk_{z}u_{t}+A_{z,\Gth}u_{\Gth}}{A_{z}A_{\Gth}(1+t\Gk_{z})}
\end{bmatrix}.
\end{equation}
The gradient restricted to the mid-surface or the so called simplified gradient denoted by $F$ is obtained from (\ref{2.3}) putting $t=0,$
thus it has the form
\begin{equation}
  \label{2.4}
\BF=
\begin{bmatrix}
  u_{t,t} & \dfrac{u_{t,\Gth}-A_{\Gth}\Gk_{\Gth}u_{\Gth}}{A_{\Gth}} &
\dfrac{u_{t,z}-A_{z}\Gk_{z}u_{z}}{A_{z}}\\[3ex]
u_{\Gth,t}  &
\dfrac{A_{z}u_{\Gth,\Gth}+A_{z}A_{\Gth}\Gk_{\Gth}u_{t}+A_{\Gth,z}u_{z}}{A_{z}A_{\Gth}} &
\dfrac{A_{\Gth}u_{\Gth,z}-A_{z,\Gth}u_{z}}{A_{z}A_{\Gth}}\\[3ex]
u_{ z,t}  & \dfrac{A_{z}u_{z,\Gth}-A_{\Gth,z}u_{\Gth}}{A_{z}A_{\Gth}} &
\dfrac{A_{\Gth}u_{z,z}+A_{z}A_{\Gth}\Gk_{z}u_{t}+A_{z,\Gth}u_{\Gth}}{A_{z}A_{\Gth}}
\end{bmatrix}.
\end{equation}
The simplified gradient $\BF$ will be very useful in our analysis as it has the following features: it is simpler than the
usual gradient $\nabla\Bu$ and it is an approximation of $\nabla\Bu$ to the order of $h$ due to the smallness of the variable $t.$
In this paper all norms $\|\cdot\|$ are $L^{2}$ norms and the Cartesian $L^2$ inner product of two functions
$f,g\colon\Omega\to\mathbb R$ will be given by
\[
(f,g)_{\Omega}=\int_{\Omega}A_zA_\Gth f(t,\Gth,z)g(t,\Gth,z)d\Gth dzdt,
\]
which gives rise to the norm $\|f\|_{L^2(\Omega)}$.
Denote by $\mathrm S$ the mid-surface of the shell.
The following conditions $(i)$ and $(ii)$ follow from the fact that the surface $S$ is $C^3,$ while conditions $(iii)$ and $(iv)$ will be assumed throughout this work.
\begin{itemize}
\item[(i)] The gradients $\nabla A_\Gth$ and $\nabla A_z$ are bounded on $\mathrm S,$ i.e.,
$|\nabla A_\Gth|, |\nabla A_z|\leq B$ for some $B>0.$
\item[(ii)] The gradients of the curvatures $\kappa_\Gth$ and $\kappa_z$ are bounded on $\mathrm S,$ i.e.,  $|\nabla \kappa_\Gth|, |\nabla \kappa_z|\leq K_1.$
\item[(iii)] The functions $A_\Gth$ and $A_z$ are uniformly positive and bounded on $\mathrm S,$ i.e., $0<a\leq A_\Gth,A_z\leq A<\infty$
for some $a$ and $A.$ This condition means that the mid-surface is non-degenerate.
\item[(iv)] The curvatures $\kappa_\Gth$ and $\kappa_z$ are bounded on $\mathrm S$ as follows: $0<k\leq |\kappa_\Gth|, |\kappa_z|\leq K.$
\end{itemize}
To make it easier for reference, we combine all the inequalities in one equation:

\begin{align}
\label{2.5}
&0<a\leq A_\Gth,A_z\leq A,\quad |\nabla A_\Gth|, |\nabla A_z|\leq B,\\ \nonumber
&0<k\leq |\kappa_\Gth|, |\kappa_z|\leq K,\\ \nonumber
&|\nabla \kappa_\Gth|, |\nabla \kappa_z|\leq K_1.
\end{align}

\section{Main results}
\label{sec:3}
\setcounter{equation}{0}
Before formulating the results we have to set the boundary conditions the displacement $\Bu\colon\Omega\to\mathbb R^3$ is going to satisfy. We mimic the situation when the shell is a vase-like shell, i.e., it has no boundary in the $\Gth$ principal direction and is cut at the principal lines in the $z$ principal direction. As mentioned in the introduction, when the load is applied to the top ($z=1+l$) of the shell and the bottom ($z=1$) is fixed, then zero boundary conditions must be imposed on some components of the displacement $\Bu$ at $z=1+l$ and $z=1.$ In the $\Gth$ direction there are no boundary conditions imposed explicitly, but let us mention, that actually there are some and the condition is implicit, namely it is periodic Dirichlet boundary conditions on the faces $\Gth=0$ and $\Gth=\omega.$ To keep things more general, we do not assume in this paper that the shell has no boundary in the $\Gth$ direction, but we rather impose boundary conditions on the $\Gth=0$ and $\Gth=\omega$
thin faces of the shell too, namely we consider the following subspace of $H^1(\Omega):$
\begin{align}
\label{3.1}
V&=\{\BGf\in W^{1,2}(\Omega;\mathbb R^3):\phi_{\Gth}(t,\Gth,1)=\phi_{z}(t,\Gth,1)=
\phi_{\Gth}(t,\Gth,1+l)=\phi_{z}(t,\Gth,1+l)=0\}\\ \nonumber
&\cap\{\BGf\in W^{1,2}(\Omega;\mathbb R^3): \phi_\Gth,\phi_z \ \text{are} \ \omega-\text{periodic in} \ \Gth\}.
\end{align}

\begin{remark}
\label{rem:3.0}
It is also worth mentioning, that the analysis in this paper works for any shells (not necessarily cut in the principal directions) with
a wide variety of boundary conditions such as Robin ones too, which can be checked easily. For instance instead of (\ref{3.1}) one can assume that the conditions $\phi_{\Gth}(t,\Gth,z)=\phi_{z}(t,\Gth,z)=0$ for all $t\in [-h/2,h/2]$ hold on the thin faces of the shell, i.e., Dirichlet boundary conditions on the in-plane components of $\BGf.$
\end{remark}
Next, let us mention that in what follows the constants $C,c,C_i,c_i>0$ will depend only on the quantities $K,$ $K_1,$ $k,$ $A,$ $a$ and $B,$ i.e., the shell mid-surface parameters. As already mentioned in the abstract, we will work in the case when the
Gaussian curvature $K_G=\kappa_\Gth \kappa_z$ has a constant sign. Our main results are sharp Korn's first and second inequalities providing Ansatz free lower bounds for displacements $\Bu$ satisfying the boundary conditions (\ref{3.1}). We remark, that the Korn-Poincar\'e inequality proven in Section~\ref{sec:4} as an auxiliary lemma can play a crucial role in the derivation of the analogous nonlinear geometric rigidity estimates in order to introduce "artificial" boundary conditions. This being said, Lemma~\ref{lem:4.1} can be viewed as another central result of this work.
\begin{theorem}[Korn's second inequality]
\label{th:3.1}
Assume that the Gaussian curvature $K_G$ has a constant sign on the mid-surface $S$, i.e., the principal curvatures $\Gk_\Gth$ and $\Gk_z$ never vanish on $S.$ Then there exists a constat $C,$ such that Korn's second inequality holds:
\begin{equation}
  \label{3.2}
\|\nabla\Bu\|^2\leq C\left(\frac{\|e(\Bu)\|\cdot \|u_t\|}{h}+\|u_t\|^2+\|e(\Bu)\|^2\right),
\end{equation}
for all $h\in(0,1)$ and $\Bu\in V.$
\end{theorem}

A useful remark concerning uniform Korn-Poincar\'e inequalities is as follows:
\begin{remark}[Uniform Korn-Poincar\'e inequality]
\label{rem:3.2}
Under the tangential boundary conditions $\Bu\cdot \Bn=0$ on one of the faces $S^{\pm}=\Br\pm \frac{h}{2}\Bn,$ the uniform Korn-Poincare inequality holds:
\begin{equation}
  \label{3.3}
\|\Bu\|_{H^1(\Omega)}\leq C\|e(\Bu)\|_{L^2(\Omega)},
\end{equation}
for all $h\in(0,1)$ and $\Bu\in V.$
\end{remark}

\begin{proof}
The proof is a direct consequence of (\ref{3.2}), the Korn-Poincar\'e inequality (\ref{4.1}),(\ref{4.1.5}), and the Poincar\'e inequality in the normal direction: $\|u_t\|\leq h\|u_{t,t}\|\leq h\|e(\Bu)\|,$ which follows from the fact that $u_t=\Bu\cdot \Bn=0$ on one of the faces $S^{\pm}=\Br\pm \frac{h}{2}\Bn.$
\end{proof}

\begin{theorem}[Korn's first inequality]
\label{th:3.2}
Assume that the Gaussian curvature $K_G$ has a constant sign on the mid-surface $S$, i.e., the principal curvatures $\Gk_\Gth$ and $\Gk_z$ never vanish on $S.$ Then the following are true:
\begin{itemize}
\item[(i)] If $K_G>0$, then there exist constants $C_0,C>0,$ such that if $l<C_0$ then
\begin{equation}
  \label{3.4}
\|\nabla\Bu\|^2\leq \frac{C}{h}\|e(\Bu)\|^2,
\end{equation}
for all $h\in(0,1)$ and $\Bu\in V.$
\item[(ii)] If $K_G<0,$ then there exists a constant $C>0,$ such that
\begin{equation}
  \label{3.5}
\|\nabla\Bu\|^2\leq \frac{C}{h^{4/3}}\|e(\Bu)\|^2,
\end{equation}
for all $h\in(0,1)$ and $\Bu\in V.$
\end{itemize}
\end{theorem}
Let us mention that the condition $l<C_0$ is purely of technical character and we believe that it can be removed, which
will be task for future. Finally, we prove the sharpness of the estimates (\ref{3.2}), (\ref{3.4}) and (\ref{3.5}).
\begin{theorem}[Existence of Ans\"atze]
\label{th:3.3}
The exponents of $h$ in the inequalities (\ref{3.2}), (\ref{3.4}) and (\ref{3.5}) are sharp, i.e., for each of them (the cases $K_G>0$, $K_G<0$) there exists a displacement $\Bu\in V$ realizing the asymptotics of $h$ in the corresponding inequality.
\end{theorem}

\section{The Korn-Poincar\'e inequality}
\label{sec:4}
\setcounter{equation}{0}
One of the key estimates in the proof of Theorem~\ref{th:3.2} is the following Korn-Poincar\'e inequality for the simplified gradient.
\begin{lemma}
\label{lem:4.1}
There exists a constant $C,$ such that the Korn-Poincar\'e inequality holds in the corresponding case:

\begin{itemize}
\item[(i)] If $K_G<0,$ then
\begin{equation}
\label{4.1}
\|u_\Gth\|,\|u_z\|\leq C\|e(\BF)\|,
\end{equation}
for all $h\in(0,1)$ and $\Bu\in V.$
\item[(ii)] If $K_G>0,$ then there exists a constant $C_0>0,$ such that if $l<C_0,$ then
\begin{equation}
\label{4.1.5}
\|\Bu\|\leq C\|e(\BF)\|.
\end{equation}
for all $h\in(0,1)$ and $\Bu\in V.$
\end{itemize}
\end{lemma}
\begin{proof}
\textbf{The Case $K_G<0.$}
The strategy here is to freeze the variable $t$ and to work solely with the lower right block of the symmetrized matrix $e(\BF).$ We eliminate the component $u_t$ from the $22$ and $33$ entries of the symmetrized simplified gradient $e(\BF),$ getting a suitable identity only in terms of $u_\Gth,$ $u_z$ and $e(\BF),$ which will allow us to estimate the norms of $u_\Gth$ and $u_z.$ We fix any $t\in\left[-h/2,h/2\right]$ and work with the functions $u_t,u_\Gth,u_z$ as functions of two variables $\Gth$ and $z.$ At the end we will integrate the obtained inequalities in $t$ to get estimates in terms of the full norm in $\Omega.$ Assume now $\varphi=\varphi(z)\in C^1(\overline{S})$ is a smooth function depending only on $z$ yet to be chosen. We have that
$$(e(\BF)_{22}, \varphi\Gk_zu_z)=\int_\Omega\varphi A_z\Gk_zu_zu_{\Gth,\Gth}+\int_\Omega\varphi A_\Gth A_z\Gk_\Gth\Gk_z u_zu_t+
\int_\Omega\varphi A_{\Gth,z} \Gk_z u_z^2,$$
and
$$(e(\BF)_{33}, \varphi\Gk_\Gth u_z)=\int_\Omega\varphi A_\Gth \Gk_\Gth u_zu_{z,z}+\int_\Omega\varphi A_\Gth A_z\Gk_\Gth\Gk_z u_zu_t+
\int_\Omega\varphi A_{z,\Gth} \Gk_\Gth u_\Gth u_z,$$
thus we get subtracting
\begin{equation}
\label{4.2}
(e(\BF)_{22},\varphi\Gk_zu_z)-(e(\BF)_{33},\varphi\Gk_\Gth u_z)=I_1+I_2-I_3-I_4,
\end{equation}
where
$$I_1=\int_\Omega\varphi A_z\Gk_zu_zu_{\Gth,\Gth},\ \ I_2=\int_\Omega\varphi A_{\Gth,z} \Gk_z u_z^2,\ \ I_3=\int_\Omega\varphi A_\Gth \Gk_\Gth u_zu_{z,z}, \ \ I_4=\int_\Omega\varphi A_{z,\Gth} \Gk_\Gth u_\Gth u_z.
$$
Next we have by integration by parts,
\begin{align}
\label{4.3}
I_1&=\int_\Omega\varphi A_z\Gk_zu_zu_{\Gth,\Gth}\\ \nonumber
&=-\int_\Omega u_{\Gth}\frac{\partial}{\partial\Gth}(\varphi A_z\Gk_zu_z)\\ \nonumber
&=-\int_\Omega u_{\Gth}u_z \frac{\partial}{\partial\Gth}(\varphi A_z\Gk_z)-I_5,
\end{align}
where
$$I_5=\int_\Omega u_{\Gth}u_{z,\Gth}(\varphi A_z\Gk_z).$$
We have furthermore,
$$I_3=-I_3-\int_\Omega\frac{\partial}{\partial z}(\varphi A_\Gth \Gk_\Gth )u_z^2,$$
thus we get
\begin{equation}
\label{4.4}
I_3=-\frac{1}{2}\int_\Omega\frac{\partial}{\partial z}(\varphi A_\Gth \Gk_\Gth )u_z^2.
\end{equation}
Let us now manipulate the summand $I_5.$ We utilize the equality,
$$A_zu_{z,\Gth}=2A_\Gth A_ze(\BF)_{23}+A_{\Gth,z}u_\Gth+A_{z,\Gth}u_z-A_\Gth u_{\Gth,z}$$
to calculate $I_5$ as follows:
\begin{align}
\label{4.5}
I_5&=\int_\Omega u_{\Gth}u_{z,\Gth}(\varphi A_z\Gk_z)\\ \nonumber
&=\int_\Omega \varphi \Gk_zu_{\Gth}(2A_\Gth A_ze(\BF)_{23}+A_{\Gth,z}u_\Gth+A_{z,\Gth}u_z-A_\Gth u_{\Gth,z})\\ \nonumber
&=2\int_\Omega \varphi \Gk_zA_\Gth A_ze(\BF)_{23}u_{\Gth}+\int_\Omega \varphi \Gk_z A_{\Gth,z}u_{\Gth}^2+
\int_\Omega \varphi \Gk_zA_{z,\Gth}u_{\Gth}u_z- \int_\Omega \varphi \Gk_z  A_\Gth u_{\Gth}u_{\Gth,z}.
\end{align}
We have integrating by parts, that
$$
\int_\Omega \varphi \Gk_z  A_\Gth u_{\Gth}u_{\Gth,z}=-\frac{1}{2}\int_\Omega \frac{\partial}{\partial z}(\varphi \Gk_z A_\Gth)u_{\Gth}^2,
$$
thus we obtain from (\ref{4.5}),
\begin{equation}
\label{4.6}
I_5=2\int_\Omega \varphi \Gk_zA_\Gth A_ze(\BF)_{23}u_{\Gth}+\int_\Omega \varphi \Gk_z A_{\Gth,z}u_{\Gth}^2+
\int_\Omega \varphi \Gk_zA_{z,\Gth}u_{\Gth}u_z+\frac{1}{2}\int_\Omega \frac{\partial}{\partial z}(\varphi \Gk_z A_\Gth)u_{\Gth}^2.
\end{equation}
Finally, we get for $I_1$ from (\ref{4.3}) the equality
\begin{equation}
\label{4.7}
I_1=-2\int_\Omega \varphi \Gk_zA_\Gth A_ze(\BF)_{23}u_{\Gth}
-\int_\Omega \left(\varphi \Gk_zA_{z,\Gth}+\frac{\partial}{\partial\Gth}(\varphi A_z\Gk_z)\right)u_{\Gth}u_z
-\int_\Omega \left(\frac{1}{2}\frac{\partial}{\partial z}(\varphi \Gk_z A_\Gth)+\varphi \Gk_z A_{\Gth,z}\right)u_{\Gth}^2
\end{equation}
and thus combining the equalities (\ref{4.2}), (\ref{4.4}) and (\ref{4.7}), we obtain the formula
\begin{align}
\label{4.8}
(e(\BF)_{22},\varphi\Gk_zu_z)&-(e(\BF)_{33},\varphi\Gk_\Gth u_z)+2(e(\BF)_{23}, \varphi\Gk_zu_\Gth)\\ \nonumber
=&-\int_\Omega \left(\frac{1}{2}\frac{\partial}{\partial z}(\varphi \Gk_z A_\Gth)+\varphi \Gk_z A_{\Gth,z}\right)u_{\Gth}^2
+\int_\Omega\left(\varphi A_{\Gth,z} \Gk_z+\frac{1}{2}\frac{\partial}{\partial z}(\varphi A_\Gth \Gk_\Gth )\right)u_z^2\\ \nonumber
&-\int_\Omega \left(\varphi A_{z,\Gth}(\Gk_\Gth+\Gk_z)+\frac{\partial}{\partial\Gth}(\varphi A_z\Gk_z)\right)u_{\Gth}u_z\\ \nonumber
=&Q(u_\Gth,u_z),
\end{align}
where $Q(v,w)$ is a quadratic form in $v$ and $w.$ Next, we aim to choose the function $\varphi$ such that the form
$Q$ is either positive or negative definite. To that end we choose $\varphi(z)=e^{\lambda z},$ where $\lambda>0$ is a big enough constant. With this choice of $\varphi,$ the coefficient of the function $u_z^2$ in the integrand in $Q(u_\Gth,u_z)$ will become
$$e^{\lambda z}\left(\frac{1}{2}\lambda A_\Gth\Gk_\Gth+A_{\Gth,z}\left(\frac{\Gk_\Gth}{2}+\Gk_z\right)+A_{\Gth}\Gk_{\Gth,z}\right),$$
hence due to the bounds (\ref{2.5}), for bing enough $\lambda$ it will have the sign of $\Gk_\Gth$ and bigger than $\frac{1}{3}e^{\lambda z}\lambda A_\Gth\Gk_\Gth$ in the absolute value uniformly on $S.$ Similarly the coefficient of $u_\Gth^2$ will have the opposite sign of $\Gk_z$ and will be bigger than $\frac{1}{3}e^{\lambda z}\lambda A_\Gth\Gk_z$ in the absolute value uniformly on $S.$ Finally, as the function $\varphi$ does not depend on $\Gth,$ then the coefficient of $u_\Gth u_z$ will be uniformly bounded by $Ce^{\lambda z}$. We choose now the parameter $\lambda $ such that
$$\frac{1}{3}\lambda A_\Gth|\Gk_z|, \ \frac{1}{3}\lambda A_\Gth|\Gk_\Gth|>2C,$$ to ensure the estimate
\begin{equation}
\label{4.9}
|Q(u_\Gth,u_z)|\geq C_1 (\|u_\Gth\|^2+\|u_z\|^2),
\end{equation}
for some constant $C_1>0.$ Therefore, we have by the Schwartz inequality owing to the estimates (\ref{4.8}) and (\ref{4.9}), that
\begin{equation}
\label{4.10}
C_2\|e(\BF)\|(\|u_\Gth\|+\|u_z\|)\geq |Q(u_\Gth,u_z)|\geq C_1 (\|u_\Gth\|^2+\|u_z\|^2),
\end{equation}
which gives the desired estimate (\ref{4.1}). The proof of the case $K_G<0$ is finished now.\\

\textbf{The Case $K_G>0.$} The proof of this case is more technical and consists of several steps. We again freeze the variable $t$ and work solely on the lower right block of the simplified gradient $\BF.$ It will become clear from the proof, that basically the structure of the coefficients of the displacement components and their partial derivatives, that comes from the fact that one is dealing with a full gradient in the lower right block, does not play an important role, but we rather use only some partial information on them such as positivity and boundedness. The proof is divided into 3 steps that are done below.\\
\textbf{Step 1: Bounding $\|u_\Gth\|$ and $\|u_z\|$ in terms of each other.} \textit{In the first step we prove that there exist two constants $c,C_1>0,$ such that}
\begin{equation}
\label{4.12}
\|u_\Gth\|^2\leq C_1e^{cl}(\|u_z\|^2+\|e(\BF)\|^2),\qquad \|u_z\|^2\leq C_1e^{cl}(\|u_\Gth\|^2+\|e(\BF)\|^2).
\end{equation}
For the proof we follow the calculation of the previous case, namely we recall the identity (\ref{4.8}):
\begin{align}
\label{4.13}
(&e(\BF)_{22},\varphi\Gk_zu_z)-(e(\BF)_{33},\varphi\Gk_\Gth u_z)+2(e(\BF)_{23}, \varphi\Gk_zu_\Gth)
+\int_\Omega \left(\frac{1}{2}\frac{\partial}{\partial z}(\varphi \Gk_z A_\Gth)+\varphi \Gk_z A_{\Gth,z}\right)u_{\Gth}^2\\ \nonumber
&=\int_\Omega\left(\varphi A_{\Gth,z} \Gk_z+\frac{1}{2}\frac{\partial}{\partial z}(\varphi A_\Gth \Gk_\Gth )\right)u_z^2
-\int_\Omega \left(\varphi A_{z,\Gth}(\Gk_\Gth+\Gk_z)+\frac{\partial}{\partial\Gth}(\varphi A_z\Gk_z)\right)u_{\Gth}u_z.
\end{align}
We are working in the case $K_G=\Gk_\Gth\Gk_z>0,$ thus $\Gk_\Gth$ and $\Gk_z$ have the same sign, which determines the sign of the coefficient of $u_\Gth^2$ and $u_z^2$ in the integrand on the left hand side and on the right hand side of (\ref{4.13}) respectively.
We have by the bounds (\ref{2.5}) and the Schwartz inequality, that there exists a constant $C>0$ such, that
\begin{align}
\label{4.14}
|(e(\BF)_{22},\varphi\Gk_zu_z)-(e(\BF)_{33},\varphi\Gk_\Gth u_z)+2(e(\BF)_{23}, \varphi\Gk_zu_\Gth)|&\leq \|\sqrt{\varphi}u_\Gth\|^2+\|\sqrt{\varphi}u_z\|^2+C\|\sqrt{\varphi}e(\BF)\|^2\\ \nonumber
\left|\int_\Omega \left(\varphi A_{z,\Gth}(\Gk_\Gth+\Gk_z)+\frac{\partial}{\partial\Gth}(\varphi A_z\Gk_z)\right)u_{\Gth}u_z\right|
&\leq C(\|\sqrt{\varphi}u_\Gth\|^2+\|\sqrt{\varphi}u_z\|^2).
\end{align}
We have like the previous case that,
$$\varphi'(z)=\lambda\varphi(z),$$
thus in order to make the coefficients of $u_\Gth^2$ and $u_z^2$ in (\ref{4.13}) uniformly bounded from below by $\varphi(z),$ we have to choose
$\lambda=c,$ where $c>0$ is a big enough constant. On the other hand the maximal value of $\varphi(z)$ must be controlled
by its minimal value, so that $\varphi$ can be
eliminated from the inequality resulting from (\ref{4.13}) and (\ref{4.14}). We have, that
\begin{equation}
\label{4.14.2}
\frac{\max\varphi(z)}{\min\varphi(z)}=e^{cl},
\end{equation}
thus combining (\ref{4.13}) and (4.14) we obtain the desired estimate (\ref{4.12}) via the observation done in the case
$K_G<0.$ \\
\textbf{Step 2: Bounding $\|u_t\|$ in terms of $\|u_\Gth\|$ and $\|u_z\|.$} \textit{In the second step we prove that there exists a constant $C_2>0,$ such that}
\begin{equation}
\label{4.15}
\|u_t\|^2\leq C_2(\|u_\Gth\|^2+\|u_z\|^2+\|e(\BF)\|^2).
\end{equation}
We aim to get an identity involving the component $u_t.$
Namely, we have on one hand that
\begin{align*}
(e(\BF)_{22}-\Gk_\Gth u_t, e(\BF)_{33}-\Gk_z u_t)&=(e(\BF)_{22}, e(\BF)_{33})-(\Gk_\Gth u_t,e(\BF)_{33})\\
&-(\Gk_z u_t,e(\BF)_{22})+(\Gk_\Gth u_t,\Gk_z u_t),
\end{align*}
thus we get by the Schwartz inequality and the positivity of the Gaussian curvature $K_G=\Gk_\Gth\Gk_z,$ that
\begin{equation}
\label{4.16}
(e(\BF)_{22}-\Gk_\Gth u_t, e(\BF)_{33}-\Gk_z u_t)\geq C_3\|u_t\|^2-C_4\|e(\BF)\|^2,
\end{equation}
for some constants $C_3,C_4>0.$ We have on the other hand, that
\begin{align}
\label{4.17}
(&e(\BF)_{22}-\Gk_\Gth u_t, e(\BF)_{33}-\Gk_zu_t)
=\int_\Omega\left(u_{\Gth,\Gth}+\frac{A_{\Gth,z}}{A_\Gth}u_z\right)\left(u_{z,z}+\frac{A_{z,\Gth}}{A_z}u_\Gth\right)\\ \nonumber
&=I_1+I_2+I_3+I_4,
\end{align}
where
\begin{equation}
\label{4.18}
I_1=\int_\Omega u_{\Gth,\Gth}u_{z,z},\ \ I_2=\int_\Omega\frac{A_{\Gth,z}A_{z,\Gth}}{A_\Gth A_z}u_zu_\Gth,\ \
I_3=\int_\Omega\frac{A_{\Gth,z}}{A_\Gth}u_zu_{z,z},\ \ I_4=\int_\Omega\frac{A_{z,\Gth}}{A_z}u_\Gth u_{\Gth,\Gth}
\end{equation}
Next we calculate by integration by parts,
\begin{equation}
\label{4.19}
I_1=\int_\Omega u_{\Gth,z}u_{z,\Gth},\quad I_3=-\frac{1}{2}\int_\Omega\frac{\partial}{\partial z}\left(\frac{A_{\Gth,z}}{A_\Gth}\right)u_z^2,\quad
I_4=-\frac{1}{2}\int_\Omega\frac{\partial}{\partial\Gth}\left(\frac{A_{z,\Gth}}{A_z}\right)u_\Gth^2.
\end{equation}
We have again by integrating by parts,
\begin{align}
\label{4.20}
(F_{23}, 2e(\BF)_{23}-F_{23})&=2(e(\BF)_{23},F_{23})-\|F_{23}\|^2\\ \nonumber
&=\int_\Omega\left(u_{z,\Gth}-\frac{A_{\Gth,z}}{A_\Gth}u_\Gth\right)\left(u_{\Gth,z}-\frac{A_{z,\Gth}}{A_z}u_z\right)\\ \nonumber
&=\int_\Omega u_{\Gth,z}u_{z,\Gth}+\int_\Omega\frac{A_{\Gth,z}A_{z,\Gth}}{A_\Gth A_z}u_zu_\Gth
+\frac{1}{2}\int_\Omega\frac{\partial}{\partial z}\left(\frac{A_{\Gth,z}}{A_\Gth}\right)u_\Gth^2
+\frac{1}{2}\int_\Omega\frac{\partial}{\partial \Gth}\left(\frac{A_{z,\Gth}}{A_z}\right)u_z^2.
\end{align}
Therefore combining the identities (\ref{4.17})-(\ref{4.20}) we get,
\begin{align*}
(e(\BF)_{22}-\Gk_\Gth u_t, e(\BF)_{33}-\Gk_zu_t)-&2(e(\BF)_{23},F_{23})+\|F_{23}\|^2\\
&=-\frac{1}{2}\int_\Omega\left(\frac{\partial}{\partial z}\left(\frac{A_{\Gth,z}}{A_\Gth}\right)+\frac{\partial}{\partial \Gth}\left(\frac{A_{z,\Gth}}{A_z}\right)\right)(u_\Gth^2+u_z^2).\\
\end{align*}
Here it comes to the place, where we use the Codazzi-Gauss identities\footnote{Although, the smoothness of the mid-surface $S$ would deliver the same estimates} (\ref{2.2}), namely we get from (\ref{2.2}) and the last identity that
\begin{equation}
\label{4.21}
(e(\BF)_{22}-\Gk_\Gth u_t, e(\BF)_{33}-\Gk_zu_t)+\|F_{23}\|^2=\frac{1}{2}(\|\sqrt{K_G}u_\Gth\|^2+\|\sqrt{K_G}u_z\|^2)+2(e(\BF)_{23},F_{23}).
\end{equation}
Combining now (\ref{4.16}) and (\ref{4.21}) we get
\begin{equation}
\label{4.22}
\|F_{23}\|^2+C_3\|u_t\|^2\leq C_4\|e(\BF)\|^2+2(e(\BF)_{23},F_{23})+\frac{1}{2}(\|\sqrt{K_G}u_\Gth\|^2+\|\sqrt{K_G}u_z\|^2).
\end{equation}
Finally, by the Schwartz inequality
$$2|(e(\BF)_{23},F_{23})|\leq 2\|F_{23}\|\cdot\|e(\BF)\|\leq\|F_{23}\|^2+\|e(\BF)\|^2,$$
estimate (\ref{4.22}) immediately implies
$$
C_3\|u_t\|^2\leq (C_4+1)\|e(\BF)\|^2+\frac{1}{2}(\|\sqrt{K_G}u_\Gth\|^2+\|\sqrt{K_G}u_z\|^2),
$$
that has exactly the form of (\ref{4.15}).\\
\textbf{Step 3: Bounds via Poincar\'e inequality.} \textit{In this step we finish the proof of the lemma.}\\
We have by the Poincar\'e inequality, that
$$\left\|\frac{1}{A_z}u_{z,z}\right\|^2\geq \frac{C}{l^2}\|u_z\|^2,$$
for some constant $C>0,$ thus we get utilizing the estimates (\ref{4.12}) and (\ref{4.15}) and the triangle inequality, that
\begin{align*}
\frac{C}{l^2}\|u_z\|^2&\leq \left\|\frac{1}{A_z}u_{z,z}\right\|^2\\
&\leq C_3(\|e(\BF)_{33}\|^2+\|u_t\|^2+\|u_\Gth\|^2)\\
&\leq C_4\|e(\BF)_{33}\|^2+C_5\|u_z\|^2+C_6e^{cl}\|u_z\|^2,
\end{align*}
which then bounds the norm $\|u_z\|$ in terms of the norm $\|e(\BF)_{33}\|$ as long as
$$C>l^2(C_5+C_6e^{cl}),$$
which is clearly satisfied as long as the shell width $l$ in the $z$ principal direction is small enough.
Now, as long as one has the bound $\|u_z\|\leq C\|e(\BF)_{33}\|,$ the proof of the lemma is achieved by
virtue of the estimates (\ref{4.12}) and (\ref{4.15}). The proof of the lemma is finished now.

\end{proof}

\section{Ansatz-free lower bounds}
\label{sec:5}
\setcounter{equation}{0}
Due to the Korn-Poincar\'e inequalities proven in the previous section and for the sake of simplicity, we make a further simplifying notation
\begin{equation}
  \label{5.1}
\BF^\ast=
\begin{bmatrix}
u_{t,t} & \dfrac{u_{t,\Gth}}{A_{\Gth}} & \dfrac{u_{t,z}}{A_{z}}\\[3ex]
u_{\Gth,t}  & \dfrac{A_{z}u_{\Gth,\Gth}+A_{z}A_{\Gth}\Gk_{\Gth}u_{t}}{A_{z}A_{\Gth}} & \dfrac{u_{\Gth,z}}{A_{z}}\\[3ex]
u_{ z,t}  & \dfrac{u_{z,\Gth}}{A_{\Gth}} & \dfrac{A_{\Gth}u_{z,z}+A_{z}A_{\Gth}\Gk_{z}u_{t}}{A_{z}A_{\Gth}}
\end{bmatrix}.
\end{equation}
It is clear from the estimates (\ref{4.1}) and (\ref{4.1.5}) and the triangle inequality, that we have in both cases $(K_G>0, K_G<0)$ the estimates
\begin{equation}
  \label{5.2}
\|\BF-\BF^\ast\|\leq C\|e(\BF)\|,\quad \|e(\BF^\ast)\|\leq C\|e(\BF)\|.
\end{equation}
Next we prove the following modification of Korn's second inequality:
\begin{lemma}
\label{lem:5.1}
There exists a constant $C,$ such that the Korn's second-like inequality holds:
\begin{equation}
\label{5.3}
\|\BF^\ast\|^2\leq C\left(\frac{\|u_t\|\cdot\|e(\BF^\ast)\|}{h}+\|u_t\|^2+\|e(\BF^\ast)\|^2\right).
\end{equation}
\end{lemma}

\begin{proof}
We prove the estimate (\ref{5.3}) block by block by freezing each of the variables $t,$ $\Gth$ and $z$. We consider the three $2\times 2$ blocks of the matrix $\BF^\ast$ as follows:\\

\textbf{The block $23$.} We aim to prove the estimate
\begin{equation}
\label{5.4}
\|F_{23}^\ast\|^2+\|F_{32}^\ast\|^2\leq C(\|u_t\|^2+\|e(\BF^\ast)\|^2).
\end{equation}
We have integrating by parts that
\begin{align*}
(F_{23}^\ast,F_{32}^\ast)&=\int_\Omega u_{\Gth,z}u_{z,\Gth}\\
&=\int_\Omega u_{\Gth,\Gth}u_{z,z}\\
&=(F_{22}^\ast-\Gk_\Gth u_t, F_{33}^\ast-\Gk_z u_t),
\end{align*}
thus owing to the bounds (\ref{2.5}) we obtain
\begin{equation}
\label{5.5}
|(F_{23}^\ast,F_{32}^\ast)|\leq C(\|u_t\|^2+\|e(\BF^\ast)\|^2).
\end{equation}
We can now estimate utilizing the bound (\ref{5.5}), that
\begin{align*}
\|F_{23}^\ast\|^2+\|F_{32}^\ast\|^2&=\|F_{23}^\ast+F_{32}^\ast\|^2-2(F_{23}^\ast,F_{32}^\ast)\\
&\leq 4\|e(\BF^\ast)\|^2+C(\|u_t\|^2+\|e(\BF^\ast)\|^2),
\end{align*}
which is the desired estimate (\ref{5.4}).\\

\textbf{The block $13$.} For the block $13$ we freeze the variable $\Gth$ and deal with two-variable functions. We aim to prove the estimate
\begin{equation}
\label{5.6}
\|F_{13}^\ast\|^2+\|F_{31}^\ast\|^2\leq C\left(\frac{\|u_t\|\cdot \|e(\BF^\ast)\|}{h}+\|e(\BF^\ast)\|^2\right).
\end{equation}
The key estimate here is the following Korn-like inequality for perturbed gradients on thin rectangles.
\begin{lemma}
\label{lem:5.2}
For $L,h>0$ denote $R=(0,h)\times(0,L).$ Assume the displacement $\BU=(u(x,y),v(x,y))\in H^1(R,\mathbb R^2)$ satisfies either of the boundary conditions below in the sense of traces:
\begin{itemize}
\item[(i)] $v(x,0)=0$ for all $x\in(0,h),$
\item[(ii)] $u(x,0)=u(x,L)$ for all $x\in(0,h).$
\end{itemize}
For any function $\varphi(y)\in L^\infty(0,L)$ denote the perturbed gradient (dropping the dependence on $\BU$ and $\varphi$) as follows:
\begin{equation}
\label{5.6.1}
\BT=
\begin{bmatrix}
u_{,x} & u_{,y}\\
v_{,x} & v_{,y}+\varphi u
\end{bmatrix}.
\end{equation}
Then the following first and a half Korn inequality holds:
\begin{equation}
\label{5.7}
\|\BT\|^2\leq C\left(\frac{\|u\|\cdot \|e(\BT)\|}{h}+\|e(\BT)\|^2\right),
\end{equation}
for all $h\in(0,1),$ where $C$ depends only on the quantity $\|\varphi\|_{L^\infty}.$ Here, as always $e(\BT)=\frac{1}{2}\left(\BT+\BT^T\right)$
and the norm $\|\cdot\|$ is the $L^2(R)$ norm.
\end{lemma}

\begin{proof}
First of all let us mention that all norms are $L^2(R)$ norms in the proof, and we will drop the dependence $L^2(R)$ for the sake of brevity.
We adopt the strategy of harmonic projections [\ref{bib:Kon.Ole.1},\ref{bib:Kon.Ole.2}]. Assume $w(x,y)$ is the harmonic part of $u$ in $R,$ i.e., it is the unique solution of the Dirichlet boundary value problem
\begin{equation}
\label{5.6.2}
\begin{cases}
\triangle w(x,y)=0, & (x,y)\in R\\
w(x,y)=u(x,y), & (x,y)\in \partial R.
\end{cases}
\end{equation}
Then we have in the sense of distributions, that
\begin{align}
\label{5.6.3}
\triangle (u-w)&=\triangle u\\ \nonumber
&=u_{,xx}+u_{,yy}\\ \nonumber
&=(e_{11}(\BT))_{,x}+(u_{,y}+v_{,x})_{,y}-v_{,xy}\\ \nonumber
&=(e_{11}(\BT))_{,x}+(2e_{12}(\BT))_{,y}-(v_{,y}+\varphi(y)u)_{,x}+\varphi(y)u_{,x}\\ \nonumber
&=(e_{11}(\BT))_{,x}+(2e_{12}(\BT))_{,y}-(e_{22}(\BT))_{,x}+\varphi(y)e_{11}(\BT).
\end{align}
Thus multiplying (\ref{5.6.3}) by $u-w$ and integrating by parts over $R$ we discover
$$
-\int_{R}|\nabla(u-w)|^2=-\int_{R}\left((u-w)_{,x}(e_{11}(\BT)-e_{22}(\BT))+2(u-w)_{,y}e_{12}(\BT)-\varphi(y)e_{11}(\BT)(u-w)\right),
$$
thus we get by the Schwartz inequality,
\begin{equation}
\label{5.6.4}
\|\nabla(u-w)\|^2\leq 4\|e(\BT)\|\|\nabla(u-w)\|+\|\varphi\|_{L^\infty}\|e(\BT)\|\|u-w\|.
\end{equation}
Due to the fact that $u-w$ vanishes on the whole boundary of $R,$ we have on the other hand by the Poincar\'e inequality (not with the best constant),
\begin{equation}
\label{5.6.5}
\|u-w\|\leq h\|\nabla(u-w)\|.
\end{equation}
Therefore, combining (\ref{5.6.4}) and (\ref{5.6.5}) we obtain the estimates
\begin{equation}
\label{5.6.6}
\|\nabla(u-w)\|\leq C\|e(\BT)\|,\qquad \|u-w\|\leq Ch\|e(\BT)\|.
\end{equation}
Next we recall the following harmonic function gradient separation estimate [\ref{bib:Gra.Har.1}, Lemma~4.3, \ref{bib:Harutyunyan.1}, Theorem~1.1].
\begin{lemma}
\label{lem:5.2.1}
Let $L,h$ and $R$ be as in Lemma~\ref{lem:5.2}. Assume the function $w(x,y)\in H^1(R)$ is harmonic in $R$ and satisfies the boundary conditions $w(x,0)=w(x,L)$ for all $x\in(0,h)$ in the
sense of traces. Then the following gradient separation estimate holds:
\begin{equation}
\label{5.6.7}
\|w_{,y}\|^2\leq\frac{2\sqrt 3}{h}\|w\|\cdot\|w_{,x}\|+\|w_{,x}\|^2.
\end{equation}
\end{lemma}
Let us show, that we can apply Lemma~\ref{lem:5.2.1} to the harmonic part $w$ of $u$. If the component $u$ satisfies the boundary conditions
$(ii)$, then $w$ does so too and we are done. If instead the $v$ component satisfies the boundary conditions $(i)$, then we can extend the displacement $\BU$ to $\overline{\BU}=(\overline{u},\overline{v})$ over the doubled rectangle $(0,h)\times(-L,L)$ as follows:
$$
\begin{cases}
\overline{u}(x,y)=u(x,y), & y\in(0,L),\\
\overline{u}(x,y)=u(x,-y), & y\in(-L,0)
\end{cases},\qquad
\begin{cases}
\overline{v}(x,y)=v(x,y), & y\in(0,L),\\
\overline{v}(x,y)=-v(x,-y), & y\in(-L,0)
\end{cases}.
$$
The function $\varphi(y)$ is extended to $(-L,0)$ by the mirror reflection as well, i.e., $\varphi(y)=\varphi(-y),$ for $y\in(0,L).$
By the boundary condition $v(x,0)=0$ we get for the extended displacement $\overline{\BU}\in H^1\left((0,h)\times(-L,L)\right)$ and also
it is easy to check that

\begin{align*}
\overline{\BT}(x,y)&=
\begin{bmatrix}
u_{,x}(x,y) & u_{,y}(x,y)\\
v_{,x}(x,y) & v_{,y}(x,y)+\varphi(y)u(x,y)
\end{bmatrix},\quad\text{if}\quad y\in(0,L),\quad\\
\overline{\BT}(x,-y)&=
\begin{bmatrix}
u_{,x}(x,y) & -u_{,y}(x,y)\\
-v_{,x}(x,y) & v_{,y}(x,y)+\varphi(y)u(x,y)
\end{bmatrix},\quad\text{if}\quad y\in(0,L),
\end{align*}
thus we get the identities
\begin{equation}
\label{5.6.8}
\|\overline{\BT}\|_{L^2(2R)}^2=2\|\BT\|_{L^2(R)}^2,\quad
\|e(\overline{\BT})\|_{L^2(2R)}^2=2\|e(\BT)\|_{L^2(R)}^2,\quad
\|\overline{u}\|_{L^2(2R)}^2=2\|u\|_{L^2(R)}^2,
\end{equation}
where $2R=(0,h)\times(-L,L).$ Moreover, the new displacement $\overline{\BU}$ satisfies the boundary condition
$\overline{u}(x,-L)=\overline{u}(x,L)$ for all $x\in(0,h).$ Thus, in the case when boundary conditions $(i)$ are satisfied,
we initially extend $\BU$ and work with $\overline{\BU}$ that satisfies the boundary condition $(ii)$. Also, due
to the equalities in (\ref{5.6.8}) the obtained estimate gives exactly the estimate (\ref{5.7}).
We now have by Lemma~\ref{lem:5.2.1}, the estimates (\ref{5.6.6}), and the triangle inequality, that
\begin{align*}
\|u_{,y}\|^2&\leq 2(\|u_{,y}-w_{,y}\|^2+\|w_{,y}\|^2)\\
&\leq 2(\|\nabla(u-w)\|^2+\frac{2\sqrt 3}{h}\|w\|\cdot\|w_{,x}\|+\|w_{,x}\|^2\\
&\leq 2C^2\|e(\BT)\|^2+\frac{2\sqrt 3}{h}(\|u\|+Ch\|e(\BT)\|)(\|u_{,x}\|+\|\nabla(u-w)\|)+(\|u_{,x}\|+\|\nabla(u-w)\|)^2\\
&\leq 2C^2\|e(\BT)\|^2+\frac{2\sqrt 3}{h}(\|u\|+Ch\|e(\BT)\|)(\|e(\BT)\|+C\|e(\BT)\|)+(\|e(\BT)\|+C\|e(\BT)\|)^2,
\end{align*}
which gives the estimate
\begin{equation}
\label{5.6.7}
\|u_{,y}\|^2\leq C\left(\frac{\|u\|\cdot \|e(\BT)\|}{h}+\|e(\BT)\|^2\right),
\end{equation}
for some constant $C>0$ depending only on $\|\varphi\|_{L^\infty}.$ The $v_{,x}$ component of the gradient is estimated in terms of
the $u_{,y}$ component and $e_{12}(\BT)$ via triangle inequality. The proof of the lemma is finished now.
\end{proof}
We now fix any $\Gth\in [0,\omega]$ and consider the displacement $\BU=(u_t,A_zu_z)$ and the function
$\varphi(z)=A_z\kappa_z$ in the variables $t$ and $z$ over the thin rectangle
$R_\Gth=(-h/2,h/2)\times(1,1+l),$ which clearly satisfies the boundary conditions $(ii)$ indicated in Lemma~\ref{lem:5.2} due to the boundary conditions (\ref{3.1}).
We can calculate
$$\BT=
\begin{bmatrix}
u_{t,t} & u_{t,z}\\
A_zu_{z,t} & A_zu_{z,z}+A_z^2\kappa_z u_t+A_{z,z}u_{z}
\end{bmatrix},
$$
thus taking into account the fact, that $A\geq A_\Gth,A_z\geq a>0$ and $|A_{z,z}|\leq B$ (bounds (\ref{2.5})), we get by Lemma~\ref{lem:5.2} and the triangle inequality,
\begin{align*}
\|u_{t,z}\|_{L^2(R_\Gth)}^2+&\|u_{z,t}\|_{L^2(R_\Gth)}^2\leq C\left(\frac{\|u_t\|_{L^2(R_\Gth)}\cdot \|e(\BT)\|_{L^2(R_\Gth)}}{h}+\|e(\BT)\|_{L^2(R_\Gth)}^2\right)\\
&\leq C\left(\frac{\|u_t\|_{L^2(R_\Gth)}(\|e(\BF^\ast)\|_{L^2(R_\Gth)}+\|u_z\|_{L^2(R_\Gth)})}{h}+
\|e(\BF^\ast)\|_{L^2(R_\Gth)}^2+\|u_z\|_{L^2(R_\Gth)}^2\right).
\end{align*}
Integrating the last inequality in $\Gth$ over $[0,\omega]$ and applying the Schwartz inequality to the product term we obtain the same
estimate for the full $L^2(\Omega)$ norms:
\begin{align*}
\|u_{t,z}\|^2+&\|u_{z,t}\|^2\leq C\left(\frac{\|u_t\|(\|e(\BF^\ast)\|+\|u_z\|)}{h}+
\|e(\BF^\ast)\|^2+\|u_z\|^2\right),
\end{align*}
which implies by virtue of Lemma~\ref{lem:4.1},
$$\|u_{t,z}\|^2+\|u_{z,t}\|^2\leq C\left(\frac{\|u_t\|\cdot \|e(\BF^\ast)\|}{h}
+\|e(\BF^\ast)\|^2\right),$$
i.e, we arrive at the desired estimate (\ref{5.6}).\\

\textbf{The block $12$.} The block $12$ corresponds to freezing the variable $z$ and is made by the entries of $\BF^\ast$ with spots $11,$ $12,$ $21$ and $22.$ We aim to prove the estimate
\begin{equation}
\label{5.8}
\|F_{12}^\ast\|^2+\|F_{21}^\ast\|^2\leq C\left(\frac{\|u_t\|\cdot \|e(\BF^\ast)\|}{h}+\|e(\BF^\ast)\|^2\right).
\end{equation}
The proof is analogous to the proof of the previous case. Indeed, we fix any $z\in[0,L]$ and consider the displacement $\BU=(u_t,A_\Gth u_\Gth)$ and the function $\varphi(\Gth)=A_\Gth\kappa_\Gth$ in the variables $t$ and $\Gth$ over the thin rectangle $R_z=(-h/2,h/2)\times(0,\omega),$ which again clearly satisfies the boundary conditions indicated in Lemma~\ref{lem:5.2} due to the boundary conditions (\ref{3.1}), and thus the proof goes through.
\end{proof}
The next step is proving the following modified version of Korn's first inequality.
\begin{lemma}
\label{lem:5.3}
There exists a constant $C$ such that the following holds:
\begin{itemize}
\item[(i)] If $K_G<0,$ then
\begin{equation}
\label{5.9}
\|\BF^\ast\|^2\leq \frac{C}{h^{4/3}}\|e(\BF^\ast)\|^2,
\end{equation}
for all $h\in(0,1)$ and $\Bu\in V.$
\item[(ii)] If $K_G>0,$ then for any $l<C_0$ one has
\begin{equation}
\label{5.10}
\|\BF^\ast\|^2\leq \frac{C}{h}\|e(\BF^\ast)\|^2,
\end{equation}
for all $h\in(0,1)$ and $\Bu\in V.$
\end{itemize}
\end{lemma}

\begin{proof}
\textbf{The Case $K_G<0.$} We have by the estimates (\ref{2.5}), integration by parts and the triangle inequality, that
\begin{align*}
C_1\|u_t\|^2&\leq |(u_t, \Gk_\Gth u_t)|\\
&\leq C_2|(u_t, e(\BF_{22}^\ast))|+C_2\left|(u_t, \frac{u_{\Gth,\Gth}}{A_\Gth})\right|\\
&\leq C_2|(u_t, e(\BF_{22}^\ast))|+C_3|(u_{t,\Gth},u_{\Gth})|+C_3|(u_{t},u_{\Gth})|,
\end{align*}
thus owing to Lemma~\ref{4.1} and the Schwartz inequality we discover
$$
C_1\|u_t\|^2\leq C_4(\|u_t\|\cdot \|e(\BF^\ast)\|+\|\BF^\ast\|\cdot \|e(\BF^\ast)\|),
$$
which then itself implies the bound
\begin{equation}
\label{5.10.1}
\|u_t\|^2\leq C(\|\BF^\ast\|\cdot \|e(\BF^\ast)\|+\|e(\BF^\ast)\|^2).
\end{equation}
Finally, combining the estimates (\ref{5.10.1}) and (\ref{5.3}) we arrive at the estimate (\ref{5.9}).\\

\textbf{The Case $K_G>0.$} In this case inequality (\ref{5.10}) is a direct consequence of the estimate (\ref{5.3}) and the Korn-Poincar\'e inequality (\ref{4.1.5}). The Lemma is proven now.
\end{proof}

\begin{proof}[Proof of Theorem~\ref{th:3.2}]
The proof can now be completed in some short observations by combining the collected estimates.
We demonstrate here how it is  done in the case $K_G<0,$ and the case $ K_G>0$ is analogous.
It is straightforward to derive the estimate
\begin{equation}
\label{5.11}
\|\BF\|^2\leq \frac{C}{h^{4/3}}\|e(\BF)\|^2
\end{equation}
from the similar inequality (\ref{5.9}). Indeed, that step follows thanks to the estimates in (\ref{5.2}).
The last step is noticing, that by the definition of the simplified gradient $\BF$ we have the estimates
$$
\|\nabla \Bu-\BF\|^2\leq \min(Ch^2\|\nabla \Bu\|^2, Ch^2\|\BF\|^2),
$$
thus invoking the estimate (\ref{5.11}), the below estimates follow:
\begin{equation}
\label{5.12}
\|\nabla \Bu\|^2\leq (1+Ch^2)\|\BF\|^2,
\end{equation}
\begin{align*}
\|e(\Bu)-e(\BF)\|^2&\leq \|\nabla \Bu-\BF\|^2\\
&\leq Ch^2\|\BF\|^2\\
&\leq Ch^{2/3}\|e(\BF)\|^2,
\end{align*}
which gives
\begin{equation}
\label{5.13}
\|e(\BF)\|^2\leq (1+Ch^{2/3})\|e(\Bu)\|^2.
\end{equation}
A combination of the estimates (\ref{5.11}), (\ref{5.12}) and (\ref{5.13}) completes the proof of the theorem.

\end{proof}

\begin{proof}[Proof of Theorem~\ref{th:3.1}] The proof is analogous to the proof of Theorem~\ref{th:3.2} by utilizing Lemma~\ref{lem:5.1} and the bounds (\ref{5.2}).

\end{proof}

\section{The Ans\"atze}
\label{sec:6}
\setcounter{equation}{0}

This section contains the proof of Theorem~\ref{th:3.3}.\\
\textbf{The case $K_G<0.$} As mentioned in the introduction, the Ansatz in this case seems ingenious and is given in the book of Tovstik and Smirnov [\ref{bib:Tov.Smi.}]. The idea of the Ansatz construction is based on several important hypotheses, two of which are as follows: first, one assumes that the Ansatz depends linearly in the normal
$t$ variable, i.e.,
\begin{equation}
\label{6.1}
\Bu=\Bu_1+t\Bu_2.
\end{equation}
Next, when one substitutes the form (\ref{6.1}) of the gradient into the formula (\ref{2.4}), then one gets a representation
$e(\Bu)=e_0(\Bu_1,\Bu_2)+te_1(\Bu_1,\Bu_2)$ for the symmetrized gradient. The second hypothesis is to choose the displacements
$\Bu_1$ and $\Bu_2$ such that all the $tt$, $t\Gth$ and $tz$ components of the summand $e_0(\Bu_1,\Bu_2)$ vanish. Now, in what follows the $t$ dependence of all functions will be exclusively explicit and linear.
The fact, that the $tt$ component of $e_0(\Bu_1,\Bu_2)$ vanishes, gives
\begin{equation}
\label{6.2}
u_t=w(\Gth,z).
\end{equation}
Next, the fact that the $t\Gth$ and $tz$ components of $e_0(\Bu_1,\Bu_2)$ vanish implies, that there exist two functions $v$ and
$s,$ such that
\begin{equation}
\label{6.3}
\begin{cases}
u_t=w,\\
u_\Gth=v-t\left(\frac{w_{,\Gth}}{A_\Gth}-\kappa_\Gth v\right),\\
u_z=s-t\left(\frac{w_{,z}}{A_z}-\kappa_z s\right).
\end{cases}
\end{equation}
Assume next the function $f(\Gth,z)$ solves the transport equation\footnote{Note that, as the principal curvatures $\kappa_\Gth$ and $\kappa_z$ do not change sign, then in the case $\kappa_z<0$ the equation (\ref{6.1}) indeed reduces to a transport equation.}
\begin{equation}
\label{6.4}
\frac{\kappa_\Gth}{A_z^2}\left(\frac{\partial f}{\partial z}\right)^2+\frac{\kappa_z}{A_\Gth^2}\left(\frac{\partial f}{\partial \Gth}\right)^2=0,
\end{equation}
the solvability of which is classical, e.g., [\ref{bib:Evans}].
Denote furthermore $n(h)=\left[\frac{1}{n^{1/3}}\right],$ where $[x]$ is the integer part of $x.$
Then the choice of the functions
\begin{equation}
\label{6.5}
\begin{cases}
w=n(h)\varphi(\Gth,z)\sin(n(h)f(\Gth,z)),\\
v=A_\Gth\kappa_\Gth\frac{\varphi(\Gth,z)}{f_{,\Gth}(\Gth,z)}\cos(n(h)f(\Gth,z)),\\
s=A_z\kappa_z\frac{\varphi(\Gth,z)}{f_{,z}(\Gth,z)}\cos(n(h)f(\Gth,z)),\\
\end{cases}
\end{equation}
gives the desired Ansatz [\ref{bib:Tov.Smi.}]. Here, $\varphi(\Gth,z)$ is a smooth function supported on the mid-surface of the shell.
By calculating the gradient $\nabla\Bu,$ it is then straightforward to check that the inequalities (\ref{3.2}) and (\ref{3.5}) become asymptotically optimal as $h\to 0.$\\
\textbf{The case $K_G>0.$} In this case the construction is easier as one has to satisfy a weaker inequality. We simply choose $v=s=0$ in (\ref{6.3}) to get the form of the Ansatz
\begin{equation}
\label{6.6}
\begin{cases}
u_t=w\\
u_\Gth=-t\frac{w_{,\Gth}}{A_\Gth}\\
u_z=-t\frac{w_{,z}}{A_z},
\end{cases}
\end{equation}
where $w$ is a smooth function compactly supported on the mid-surface $S.$ We choose $w$ to have rapid oscillations in the $\Gth$ direction, i.e.,
we assume $w(\Gth,z)=W(\frac{\Gth}{h^\alpha},z).$ Doing this we make sure, that the $\Gth$ derivative of $w,$ i.e., the $t\Gth$ component of the gradient $\Bu$ is $\frac{1}{h^\alpha}$ times greater than $w.$ We finally choose $\alpha=\frac{1}{2}$ to maximize the quantity
$\frac{\|\nabla \Bu\|}{\|e(\Bu)\|}$ in terms of the asymptotics. We get in conclusion the Ansatz
\begin{equation}
\label{6.7}
\begin{cases}
u_t=W(\frac{\Gth}{\sqrt{h}},z)\\
u_\Gth=-\frac{t\cdot W_{,\Gth}\left(\frac{\Gth}{\sqrt h},z\right)}{A_\Gth{\sqrt h}}\\
u_z=-\frac{t\cdot W_{,z}\left(\frac{\Gth}{\sqrt h},z\right)}{A_z},
\end{cases}
\end{equation}
that gives the asymptotics $\frac{\|e(\Bu)\|^2}{\|\nabla \Bu\|^2}\sim h,$ $\|u_t\|\sim \|e(\Bu)\|,$ thus inequalities (\ref{3.2}) and (\ref{3.4}) become asymptotically optimal as $h\to 0.$\\

\section*{Acknowledgements.}
The author is grateful to Yury Grabovsky for many fruitful discussions, in particular
for giving hints about the exponents $4/3$ and $1$ in the main inequalities and pointing out
the book [\ref{bib:Tov.Smi.}] and the test functions in it. The author would like to also thank 
Jirayr Avetisyan for reading and spell-checking the entire manuscript.

\section*{Conflict of interests statement.}

There is no conflict of interests.

\end{document}